\documentclass[a4paper,10pt,reqno]{amsart}
\usepackage{amsmath,amsfonts,amssymb,amsthm}
\usepackage{mathtools}
\usepackage{color,xcolor}
\usepackage[margin=3cm]{geometry}
\usepackage[colorlinks=true,citecolor=blue,urlcolor=customgreen]{hyperref}
\usepackage{caption,enumerate}
\newtheorem{theorem}{Theorem}[section]

\newtheorem{corollary}[theorem]{Corollary}
\newtheorem{remark}{Remark}[section]
\numberwithin{equation}{section}

\newcommand{\CC}{\ensuremath{\mathbb{C}}}
\newcommand{\RR}{\ensuremath{\mathbb{R}}}
\newcommand{\ZZ}{\ensuremath{\mathbb{Z}}}

\newcommand{\md}{\mathrm{d}}
\newcommand{\me}{\mathrm{e}}
\newcommand{\mo}{\mathcal{O}}
\newcommand{\mi}{\mathrm{i}}

\definecolor{customgreen}{rgb}{0.0, 0.5, 0.0}

\author[P.-C. Hang]{Peng-Cheng Hang}
\address{Peng-Cheng Hang\\School of Mathematics and Statistics, Donghua University, Shanghai 201620,
	People's Republic of China.}\email{mathroc618@outlook.com}

	
\begin{document}
\title[Matrix analogues of asymptotic methods]{Matrix analogues of some asymptotic methods\\ for Laplace integrals}

\begin{abstract}
  The matrix analogues of Laplace's method and Watson's lemma are derived via the approach described by Williams and Wong [J. Approx. Theory 24 (4) (1974), 378--384]. Some examples are also given.
\end{abstract}

\maketitle

\section{\bf Introduction}
  Williams and Wong \cite{WiWo} studied the operator analogue of Watson's lemma. In their work, a key assumption for the matrix form of Watson's lemma is that the matrix is normal. Luke and Barker \cite{LuBa} tried to weaken the ``normal" condition, but the constant $M_1$ in \cite[Eq. (40)]{LuBa} is not absolute, which motivates us to find a correct answer.
  
  This paper is organised as follows. In Section \ref{Sect 2}, we deduce an operator version of Laplace's method. In Section \ref{Sect 3}, we make a revision of Luke and Barker's result \cite[Theorem 3]{LuBa}. This paper concludes with some examples in Section \ref{Sect 4}.
  
\section{\bf Operator analogue of Laplace's method}\label{Sect 2}
  In this section, we adopt the definition of asymptotic expansions of operator-valued functions introduced in \cite[Sect. 3]{WiWo} and establish the operator analogue of Laplace's method.

  Let $X$ be a Banach space over the complex field $\CC$ and $Y$ be a dense subspace of $X$. Let $A$ be a closed linear operator from $Y$ into $X$. The spectrum and resolvent of $A$ are denoted by $\sigma(A)$ and $R_{\lambda}(A)=\left(\lambda I-A\right)^{-1}$, respectively.
  
  According to \cite[Sect. 2]{WiWo}, if, in addition, $A$ satisfies the following conditions:
  \begin{itemize}
  	\item $(\text{C}_1)$ There exists a positive $\Delta$ such that
  	\[\sigma(A)\subset \left\{\lambda\in\CC:\lambda\ne 0\text{ and }\left|\arg \lambda\right|\leqslant \frac{\pi}{2}-\Delta\right\},\]
  	\item $(\text{C}_2)$ Let $\omega(A):=\inf\{\Re(\lambda):\lambda\in\sigma(A)\}$. There exists $M>0$ and $0<\omega_1\leqslant \omega(A)$ such that for any positive integer $n$,
  	\begin{equation}\label{C2 condition inequality}
  		\left\|R_{\lambda}\left(A\right)^n\right\|\leqslant M\left(\omega_1-\lambda\right)^{-n}\quad (\lambda<\omega_1),
  	\end{equation}
  \end{itemize}
  then $\me^{tA}$ is well-defined for real $t$. Then we can define for $0<\alpha<\infty$ that
  \begin{equation}\label{fractional power of operator}
  	A^{-\alpha}=\frac{1}{\Gamma(\alpha)}\int_0^{\infty}t^{\alpha-1}\me^{-tA}\md t.
  \end{equation}
  Furthermore, there exists a strongly continuous semigroup, $\left\{\me^{-tA}:0\leqslant t<\infty\right\}$, of bounded linear operators on $X$ such that
  \begin{equation}\label{semigroup inequality}
  	\left\|\me^{-tA}\right\|\leqslant M\me^{-t\omega_1},\quad \text{for all }0<t<\infty.
  \end{equation}
  
  Let us state our main results. Our proof is based on the classical procedure in \cite[p. 57-60]{Wong}.
  
  \begin{theorem}\label{Thm 2.1}
  	Let $\{A_{\alpha}\}$ be a net of closed linear operators from $Y$ into $X$, each satisfying conditions $(\text{C}_1)$ and $(\text{C}_2)$ with the same $\Delta$ and $M$ and such that there is some positive $\eta$ with $\omega_1(A_{\alpha})\geqslant \eta\,\omega(A_{\alpha})$ for each $A_{\alpha}$. Define the integral
  	\[I(z):=\int_a^b \varphi(x)\me^{-zh(x)}\md x,\]
  	where $h(x)$ is real, $a\in\RR$ is finite and $b(>a)$ is finite or infinite.
  	
  	Assume that (i) $h(x)>h(a)$ for all $x\in (a,b)$, and for every $\delta>0$ the infimum of $h(x)-h(a)$ in $[a+\delta,b)$ is positive; (ii) $h'(x)$ and $\varphi(x)$ are continuous in a neighborhood of $x=a$, except possibly at $a$; (iii) as $x\to a^+$, the following expansions hold
  	\[h(x)\sim h(a)+\sum_{s=0}^{\infty}a_s\left(x-a\right)^{s+\mu},\quad\varphi(x)\sim \sum_{s=0}^{\infty}b_s\left(x-a\right)^{s+\lambda-1},\]
  	where the expansion for $h(x)$ is differentiable, $\lambda$ and $\mu$ are positive constants, and $a_0>0$; and (iv) the integral $I(z)$ converges absolutely for all sufficiently large $z$.
  	
  	Then the bounded linear operator $I(A_{\alpha})$ has the asymptotic expansion
  	\begin{equation}\label{Laplace's method expansion}
  		I(A_{\alpha})\sim \me^{-h(a)A_{\alpha}}\sum_{s=0}^{\infty}\Gamma\left(\frac{s+\lambda}{\mu}\right)c_s A_{\alpha}^{-\frac{s+\lambda}{\mu}},\quad \text{as }\left\|A_{\alpha}^{-1}\right\|\to 0,
  	\end{equation}
  	where the coefficients $c_s$ are defined by the expansion
  	\begin{equation}\label{definition of coefficents}
  		\frac{\varphi(x)}{h'(x)}\sim \sum_{s=0}^{\infty}c_s t^{(s+\lambda-\mu)/\mu},\quad \text{as }\,t=h(x)-h(a)\to 0^+.
  	\end{equation}
  	In particular, $c_0=b_0\mu^{-1}a_0^{-\lambda/\mu}$.
  \end{theorem}
  
  \begin{proof}
  	Let $\omega_{1,\alpha}=\omega_1(A_{\alpha})$ and $\omega_{\alpha}=\omega(A_{\alpha})$. Then $\omega_{1,\alpha}^{-1}\leqslant \|A_{\alpha}^{-1}\|\left(\eta\sin\Delta\right)^{-1}$ (see \cite[Eq. (3)]{WiWo}). Thus $\omega_{1,\alpha}\to +\infty$, as $\|A_{\alpha}^{-1}\|\to 0$. In view of this, using \eqref{semigroup inequality} and the condition (iv) gives
  	\begin{align*}
  		\|I(A_{\alpha})\|&\leqslant \big\|\me^{-h(a)A_{\alpha}}\big\|\int_a^b |\varphi(x)|\left\|\me^{-(h(x)-h(a))A_{\alpha}}\right\|\md x\\
  		& \leqslant M\big\|\me^{-h(a)A_{\alpha}}\big\|\int_a^b |\varphi(x)|\cdot \me^{-\omega_{1,\alpha}(h(x)-h(a))}\md x<\infty,
  	\end{align*}
  	which confirms the boundness of $I(A_{\alpha})$ for suitable $\alpha$.
  	
  	Conditions (ii) and (iii) imply the existence of a number $c$ in $(a,b)$ such that $h'(x)$ and $\varphi(x)$ are continuous in $(a,c]$, and $h'(x)$ is positive there. Put $T=h(c)-h(a)>0$ and introduce the new variable $t=h(x)-h(a)$. Since $h(x)$ is increasing in $(a,c)$, we may write
  	\begin{equation}\label{new integral in Laplace's method}
  		\me^{h(a)A_{\alpha}}\int_a^c \varphi(x)\me^{-h(x)A_{\alpha}}\md x=\int_0^T f(t)\me^{-tA_{\alpha}}\md t
  	\end{equation}
  	with $f(t)$ beging the continuous function in $(0,T]$ given by
  	\[f(t)=\varphi(x)\frac{\md x}{\md t}=\frac{\varphi(x)}{h'(x)}.\]
  	Now $f(t)$ has the expansion \eqref{definition of coefficents} as $t\to 0^+$. Applying the operator analogue of Watson's lemma \cite[Theorem 1]{WiWo} to the integral on the right-hand side of \eqref{new integral in Laplace's method}, we have
  	\begin{equation}\label{leading expansion in Laplace's method}
  		\int_a^c \varphi(x)\me^{-h(x)A_{\alpha}}\md x\sim \me^{-h(a)A_{\alpha}}\sum_{s=0}^{\infty}\Gamma\left(\frac{s+\lambda}{\mu}\right)c_s A_{\alpha}^{-\frac{s+\lambda}{\mu}},\quad \text{as }\left\|A_{\alpha}^{-1}\right\|\to 0.
  	\end{equation}
  	
  	Now estimate the integral on the remaining interval $(c,b)$. Let $\omega_0$ be a value of $z>0$ for which $I(z)$ is absolutely convergent, ant set
  	\[\varepsilon=\inf_{c\leqslant x<b}(h(x)-h(a)).\]
  	By condition (i), $\varepsilon$ is positive. Now for $\omega_{1,\alpha}\geqslant \omega_0$,
  	\begin{align}
  		\omega_{1,\alpha}[h(x)-h(a)]& =\left(\omega_{1,\alpha}-\omega_0\right)[h(x)-h(a)]+\omega_0[h(x)-h(a)] \nonumber\\
  		& \geqslant \left(\omega_{1,\alpha}-\omega_0\right)\varepsilon+\omega_0[h(x)-h(a)].\label{h(x) inequality in Laplace's method}
  	\end{align}
  	Let $I_1(A_{\alpha}):=\int_c^b \varphi(x)\me^{-h(x)A_{\alpha}}\md x$. Applying \eqref{semigroup inequality}, \eqref{h(x) inequality in Laplace's method} and the condition (iv), we get
  	\begin{align*}
  		\big\|\me^{h(a)A_{\alpha}}I_1(A_{\alpha})\big\|& \leqslant M\int_c^b |\varphi(x)|\me^{-\omega_{1,\alpha}(h(x)-h(a))}\md x\\
  		& \leqslant M\me^{(\omega_0-\omega_{1,\alpha})\varepsilon}\int_c^b |\varphi(x)|\me^{-\omega_0(h(x)-h(a))}\md x=K\me^{-\omega_{1,\alpha}\varepsilon}.
  	\end{align*}
  	Recall that $\omega_{1,\alpha}^{-1}\leqslant \|A_{\alpha}^{-1}\|\left(\eta\sin\Delta\right)^{-1}$. Then
  	\begin{equation}\label{reminder estimate in Laplace's method}
  		\big\|\me^{h(a)A_{\alpha}}I_1(A_{\alpha})\big\|\leqslant K\me^{-C/\|A_{\alpha}^{-1}\|},
  	\end{equation}
  	where $C=\varepsilon\eta\sin\Delta$. The result \eqref{Laplace's method expansion} follows from \eqref{leading expansion in Laplace's method} and \eqref{reminder estimate in Laplace's method}.
  \end{proof}
  
  A similar discussion of \cite[Corollary]{WiWo} gives the following result.
  \begin{corollary}
  	Let $\{A_{\alpha}\}$ be a net of bounded linear operators on a Hilbert space $\mathcal{H}$. If each $A_{\alpha}$ is normal and satisfies the condition $(\text{C}_1)$ with the same $\Delta$, and if $\varphi(x)$ and $h(x)$ satisfy conditions (i), (ii), (iii) and (iv), then the expansion (\ref{Laplace's method expansion}) holds.
  \end{corollary}
  
\section{\bf The work of Luke and Barker revisited}\label{Sect 3}
  Luke and Barker \cite[Sect. II]{LuBa} showed that their result \cite[Theorem 3]{LuBa} satisfies the hypotheses of \cite[Theorem 1]{WiWo}. However, the constant $M_1$ in \cite[Eq. (40)]{LuBa} depends on the index $m$. More explicitly, we have the following fact
  \[M_1=\sum_{r=0}^{p}\left(m\right)_r\left|\lambda_1-\lambda\right|^{-r}>1+m\left|\lambda_1-\lambda\right|^{-1},\]
  which implies that under their restrictions, the inequality \eqref{C2 condition inequality} is not valid.
  
  In this section, we shall make a revision of their result. Before this, let us do some preparations.
  
  For $A\in \CC^{m\times m}$, let $A^+:=\frac{1}{2}\left(A+A^H\right)$ and define its 2-norm by
  \[\|A\|=\sup_{\|x\|_2=1}\|Ax\|_2,\]
  where for any vector $x\in\CC^{m\times 1}$, $\|x\|_2=\sqrt{x^H x}$. Denote
  \[\mu(A)=\min\big\{z\in\CC:z\in\sigma\left(A^+\right)\big\},\ \omega(A)=\min\{\Re(\lambda):\lambda\in\sigma(A)\}.\]
  By \cite[p. 647]{HiGa}, it follows that
  \begin{equation}\label{matrix norm inequality}
  	\left\|\me^{-tA}\right\|\leqslant \me^{-t\mu(A)}\quad (t\geqslant 0).
  \end{equation}
  
  Next we derive the matrix form of Watson's lemma.
  
  \begin{theorem}
  	Let $A$ be a matrix such that the condition $(\text{C}_1)$ (see Sect. \ref{Sect 2}) is satisfied and there is a positive $\eta$ with $\mu(A)\geqslant \eta\,\omega(A)$. Assume that: (a) the function $f(t)$ is a locally integrable function on $[0,\infty)$ and admits the expansion
  	\[f(t)=\sum_{n=1}^{\infty}a_n t^{n/r-1},\quad |t|\leqslant c+\delta,\]
  	where $r,c$ and $\delta$ are positive; and (b) there exist positive constants $M_0$ and $b$ independent of $t$ such that $|f(t)|<M_0 e^{bt}$ holds for $t\geqslant c$. Then
  	\begin{equation}\label{Watson's lemma expansion}
  		\int_{0}^{\infty}f(t)\me^{-tA}\md t\sim \sum_{n=1}^{\infty}a_n\Gamma(n/r)A^{-n/r},\quad\text{as }\left\|A^{-1}\right\|\to 0.
  	\end{equation}
  \end{theorem}
  
  \begin{proof}
  	Fix an integer $N\geqslant 2$. Clearly, there is a constant $C$ such that for all $t>0$,
  	\begin{equation}\label{integrand estimate in Watson's lemma}
  		\bigg|f(t)-\sum_{n=1}^{N-1}a_n t^{n/r-1}\bigg|\leqslant Ct^{N/r-1}\me^{bt}.
  	\end{equation}
  	A similar derivation of \cite[Eq. (3)]{WiWo} gives
  	\begin{equation}\label{mu estimate}
  		\mu(A)\geqslant \left\|A^{-1}\right\|^{-1}\eta\sin\Delta.
  	\end{equation}
  	 This together with \eqref{matrix norm inequality} and \eqref{fractional power of operator} shows that $A^{-n/r}$  is well-defined. Using \eqref{fractional power of operator}, we can infer that
  	\[R_N(A):=\int_{0}^{\infty}f(t)\me^{-tA}\md t-\sum_{n=1}^{N-1}a_n\Gamma(n/r)A^{-n/r}\]
  	satisfies
  	\[R_N(A)=\int_{0}^{\infty}\bigg(f(t)-\sum_{n=1}^{N-1}a_n t^{n/r-1}\bigg)\me^{-tA}\md t.\]
  	
  	If $\|A^{-1}\|\leqslant \frac{\eta\sin\Delta}{2b}$, then $\mu(A)\geqslant 2b$ by \eqref{mu estimate} and hence
  	\[\left\|R_N(A)\right\|\leqslant CM\int_0^{\infty}t^{N/r-1}\me^{(b-\mu(A))t}\md t\]
  	by virtue of \eqref{matrix norm inequality} and \eqref{integrand estimate in Watson's lemma}. Since $\mu(A)\to \infty$, it follows from \eqref{mu estimate} that
  	\[\left\|R_N(A)\right\|=\mo\left(\mu\left(A\right)^{-N/r}\right)=\mo\left(\left\|A^{-1}\right\|^{-N/r}\right).\]
  	Recall the inequality \cite[p. 382]{WiWo}
  	\[\left\|A^{-1}\right\|^{-N/r}\leqslant C_1\big\|A^{-(N-1)/r}\big\|\left\|A^{-1}\right\|^{1/r},\]
  	where $C_1=C_1(N,r)$. Thus $\left\|R_N(A)\right\|=o\left(\left\|A^{-(N-1)/r}\right\|\right)$, which completes the proof.
  \end{proof}
  
  \begin{corollary}
  	Let $A$ be a normal matrix satisfying the condition $(\text{C}_1)$ and $f(t)$ be a function satisfying the conditions (a) and (b). Then the expansion (\ref{Watson's lemma expansion}) holds.
  \end{corollary}
  
  \begin{proof}
  	It suffices to show that $\mu(A)\geqslant \eta\,\omega(A)$. Indeed, the further fact that $\mu(A)=\omega(A)$ follows from the spectral theorem for normal matrices.
  \end{proof}\vspace{0.5mm}
  
  \begin{remark}
  	~
  	\begin{itemize}
  		\item[(1)] Karamata's Tauberian theorem is a converse of Watson's lemma (see \cite[Sect. XIII.5]{Fell}) and its matrix analogue was given in \cite{MeSc}.
  		
  		\item[(2)] New approaches are needed to derive the matrix analogues of asymptotic methods for Mellin convolution integrals developed in \cite{Lope, Sidi, Wong 1979}.
  	\end{itemize}
  \end{remark}

\section{\bf Examples}\label{Sect 4}
  In this section, we deal with asymptotics of the gamma matrix function, Bessel matrix functions and matrix-valued Kummer function. Here we assume that $A$ is a matrix in $\CC^{m\times m}$ and the norm $\|A\|$ is a least upper bound on the matrices of complex numbers.

\subsection{Gamma matrix function}
  If $A$ is a matrix such that $\Re(z)>0$ for any eigenvalue $z$ of $A$, then the gamma matrix function $\Gamma(A)$ is defined as \cite[Eq. (1)]{JoCo}
  \[\Gamma(A)=\int_0^{\infty}t^{A-I}\me^{-t}\md t,\quad t^{A-I}\equiv \exp\left((A-I)\log t\right).\]
  Moreover, if $A+nI$ is invertible for all integer $n\geqslant 0$, then $\Gamma(A)$ is invertible, and the reciprocal gamma function is defined as
  \[\Gamma^{-1}(A)=A(A+I)\cdots(A+(n-1)I)\Gamma\left(A+nI\right)^{-1},\quad n\geqslant 1.\]
  Clearly, $\Gamma^{-1}(A)=\Gamma\left(A\right)^{-1}$ if $\sigma(A)\subset\CC\backslash\ZZ_{\leqslant 0}$.
  
  Now we can establish asymptotics of the gamma matrix function.
  
  \begin{theorem}\label{Thm 4.1}
  	Let $A$ be a matrix satisfying: (c) $\sigma(A)\subset \{\lambda\in\CC:\lambda\ne 0,\left|\arg(\lambda-\theta)\right|<\Delta\}$ with $0<\Delta<\frac{\pi}{2}$ and $0<\theta<2\pi$; and (d) for an invertible matrix $P$ such that $P^{-1}AP=J$ is the Jordan normal form of $A$, the number $n(P):=\|P\|\|P^{-1}\|$ is bounded as $\|A^{-1}\|\to 0$. Then
  	\begin{equation}\label{gamma matrix function expansion}
  		\Gamma(A)\sim \sqrt{2\pi}\,\me^{-A}A^{A-I/2}\sum_{k=0}^{\infty}g_k A^{-k},\quad \text{as }\left\|A^{-1}\right\|\to 0,
  	\end{equation}
  	where $g_k$ has explicit expressions. In particular, $g_0=1$ and $g_1=\frac{1}{12}$.
  \end{theorem}
  \begin{proof}
  	Note that the scaled gamma function
  	\[\Gamma^*(z):=\frac{1}{\sqrt{2\pi}}\me^{z}z^{1/2-z}\Gamma(z)\]
  	is analytic in the domain described in the condition (c). Since \eqref{gamma matrix function expansion} holds for $\Gamma^*(z)$ as $|z|\to \infty$, we can make use of \cite[Theorem 1]{LuBa} and obtain that \eqref{gamma matrix function expansion} holds for $\Gamma^*(A)$ as $\|A^{-1}\|\to 0$.
  \end{proof}
  
  \begin{theorem}
  	If $A$ is Hermitian and postive definite, then for any integer $N\geqslant 1$,
  	\[\Gamma(A)=\sqrt{2\pi}\,\me^{-A}A^{A-I/2}\bigg(\sum_{k=0}^{N-1}g_k A^{-k}+R_N(A)\bigg),\]
  	and the reminder $R_N(A)$ has the estimate
  	\begin{equation}\label{reminder estimate for gamma function}
  		\left\|R_N(A)\right\|_2\leqslant \frac{1+\zeta(N)}{\left(2\pi\right)^{N+1}}\Gamma(N)\left\|A^{-1}\right\|_2^{-N},
  	\end{equation}
  	where $\zeta$ denotes the Riemann zeta function.
  \end{theorem}
  
  \begin{proof}
  	Now take the spectral norm for matrices: $\|A\|\equiv\|A\|_2$. According to \cite[p. 100]{LuBa} (see also \cite[p. 271]{RiNa}), it holds for $\|B\|_2$ that if $B$ is Hermitian and positive definite, then any inequality satisfied by $f(z)$ must also be satisfied by $f(B)$. Thus, \eqref{reminder estimate for gamma function} follows from the error bound $R_N(z)$ for asymptotics of the scalar gamma function(see \cite{Boyd}) and the fact that $A$ is positive definite.
  \end{proof}
  
  \begin{remark}
  	It is of interest to find the matrix analogues of asymptotic methods for contour integrals, since these methods may be helpful to derive asymptotics of special matrix functions, such as the asymptotics of $\Gamma(A)$ in the general case $\sigma(A)\subset \{\lambda\in\CC:\lambda\ne 0,\left|\arg\lambda\right|<\pi-\delta\}$.
  \end{remark}

\subsection{Bessel matrix functions}
  Integral expressions of Bessel matrix functions are given by
  \begin{align*}
  	J_A(z)& =\left(\frac{z}{2}\right)^{A}\frac{\Gamma^{-1}\left(A+I/2\right)}{\sqrt{\pi}}\int_{-1}^1\left(1-t^2\right)^{A-I/2}\cos(zt)\md t,\ \ \left|\arg z\right|<\pi;\\
  	I_A(z)& =\left(\frac{z}{2}\right)^{A}\frac{\Gamma^{-1}\left(A+I/2\right)}{\sqrt{\pi}}\int_{-1}^1\left(1-t^2\right)^{A-I/2}\cosh(zt)\md t,\ \ \left|\arg z\right|<\frac{\pi}{2},
  \end{align*}
  where $z\ne 0$ and $\mu(A)>-\frac{1}{2}$; see \cite[Sect. 3]{SaJo} for details.
  
  Let $A$ be a normal matrix satisfying the condition $(\text{C}_1)$ (see Section \ref{Sect 2}) and set
  \[h(t)=-\log\left(1-t^2\right),\ L(A)=\left(\frac{z}{2}\right)^{A}\Gamma^{-1}\left(A+I/2\right)\left(A-I/2\right)^{-1/2}.\]
   Then Theorem \ref{Thm 2.1} confirms that as $\|A^{-1}\|\to 0$,
  \begin{equation}\label{Bessel function equivalence}
  	J_A(z)\sim L(A),\ z\ne 0,\left|\arg z\right|<\pi;\quad I_A(z)\sim L(A),\ z\ne 0,\left|\arg z\right|<\frac{\pi}{2}.
  \end{equation}
  Since $A$ is normal, the fact that $n(P)=1$ verifies for the spectral norm and thus $n(P)$ is bounded for any matrix norm. Hence, applying Theorem \ref{Thm 2.1} and \eqref{Bessel function equivalence}, we get that as $\|A^{-1}\|\to 0$,
  \begin{align*}
  	J_A(z)& \sim \frac{1}{\sqrt{2\pi}}\left(\frac{\me z}{2}\right)^A A^{-A-I/2},\quad z\ne 0,\left|\arg z\right|<\pi;\\
  	I_A(z)& \sim \frac{1}{\sqrt{2\pi}}\left(\frac{\me z}{2}\right)^A A^{-A-I/2},\quad z\ne 0,\left|\arg z\right|<\frac{\pi}{2}.
  \end{align*}
  
\subsection{Matrix-valued Kummer function}
  Similar to the scalar case, the matrix-valued confluent hypergoemetric function ${}_1F_1$ is defined as
  \[{}_1F_1(a;b;A):=\sum_{n=0}^{\infty}\frac{\left(a\right)_n}{\left(b\right)_n}\frac{A^n}{n!},\quad A\in\CC^{m\times m}.\]
  Moreover, we have the matrix analogue of Kummer's first formula \cite[p. 125, Eq. (2)]{Rain}
  \begin{equation}\label{Kummer first formula}
  	{}_1F_1(a;b;A)=\me^A\,_1F_1(b-a;b;-A).
  \end{equation}
  
  Let $A$ be a normal matrix satisfying the condition $(\text{C}_1)$. Then $\mu(A)=\omega(A)$ and further,
  \begin{equation}\label{eigenvalue identity}
  	\lambda_i(A^+)=\Re(\lambda_i(A)),\quad 1\leqslant i\leqslant n,
  \end{equation}
  where $A^+=\frac{1}{2}\left(A+A^H\right)$ and $\lambda_i(A)$ is the $i$-th eigenvalue of $A$.
  
  Write $B=-\mi\cdot\log A$ and then $A^{\mi t}=\me^{-tB}$. Now combining \eqref{eigenvalue identity} with \eqref{matrix norm inequality} and the condition $(\text{C}_1)$ gives the following estimate
  \begin{equation}\label{estimate for imaginary matrix power}
  	\big\|A^{\mi t}\big\|\leqslant \me^{(\frac{\pi}{2}-\Delta)|t|},\quad t\in\RR.
  \end{equation}
  An immediate result of the moment inequality \cite[p. 115]{Krei} is
  \[\big\|A^{-r}\big\|\leqslant C_r\big\|A^{-1}\big\|^r,\quad r>0.\]
  Then for any positive integer $N$,
  \begin{equation}\label{error estimate for negative matrix power}
  	\big\|A^{-N-1/2}\big\|\leqslant \big\|A^{-N}\big\|\big\|A^{-1/2}\big\|\leqslant C\big\|A^{-N}\big\|\big\|A^{-1}\big\|^{1/2}.
  \end{equation}
  
  Applying the estimates \eqref{estimate for imaginary matrix power} and \eqref{error estimate for negative matrix power} to the classical procedures in \cite[p. 107-108 and 193-195]{PaKa}, we may readily obtain the Mellin-Barnes integral representation
  \[\frac{\Gamma(a)}{\Gamma(b)}\,_1F_1(a;b;-A)=\frac{1}{2\pi\mi}\int_{c-\mi\infty}^{c+\mi\infty}\frac{\Gamma(-s)\Gamma(s+a)}{\Gamma(s+b)}A^s\md s\]
  and the asymptotic expansion
  \[{}_1F_1(a;b;-A)\sim \frac{\Gamma(b)}{\Gamma(b-a)}A^{-a}\sum_{n=0}^{\infty}\frac{\left(a\right)_n\left(a-b+1\right)_n}{n!}A^{-n},\quad \text{as }\left\|A^{-1}\right\|\to 0.\]
  Due to \eqref{Kummer first formula}, we also have the following expansion
  \[{}_1F_1(a;b;A)\sim\frac{\Gamma(b)}{\Gamma(a)}\me^A A^{a-b}\sum_{n=0}^{\infty}\frac{\left(c-a\right)_n\left(1-a\right)_n}{n!}A^{-n},\quad \text{as }\left\|A^{-1}\right\|\to 0.\]
  

\end{document}